\documentclass{article}
\usepackage{amsmath}
 \usepackage{latexsym,amsmath,enumerate,amssymb,amsbsy,amsthm} 
 \usepackage{enumerate,amssymb,amsmath,mathrsfs,latexsym,
 	amsfonts,txfonts,graphicx,multicol,verbatim,amsthm}

\theoremstyle{plain}

\theoremstyle{definition}

\newtheorem{thm}{Theorem}
\newtheorem{lem}[thm]{Lemma}
\newtheorem{cor}[thm]{Corollary}
\newtheorem{prop}[thm]{Proposition}

\newtheorem{ex}[thm]{Example} 
\newtheorem{rem}[thm]{Remark}

\numberwithin{equation}{section}
\numberwithin{thm}{section}

\newcommand{\N}{\ensuremath{\mathbb{N}}}
\newcommand{\Z}{\ensuremath{\mathbb{Z}}}
\newcommand{\Fq}{\ensuremath{\mathbb{F}_q}}
\newcommand{\Fqq}{\ensuremath{\mathbb{F}_{q^2}}}
\newcommand{\F}{\ensuremath{\mathbb{F}}}
\newcommand{\al}{\ensuremath{\alpha}}
\newcommand{\be}{\ensuremath{\beta}}
\newcommand{\gam}{\ensuremath{\gamma}}
\newcommand{\lam}{\ensuremath{\lambda}}

\title{Constacyclic and Quasi-Twisted Hermitian Self-Dual  Codes over Finite Fields\thanks{This research is supported by the Thailand Research Fund  under Research Grant TRG5780065.}}

\author{Ekkasit~Sangwisut$^a$, Somphong~Jitman$^{b,}$\thanks{Corresponding author.}   ~and Patanee Udomkavanich$^c$}

\date{}
\begin{document}
\maketitle


{\small
	\centerline{$^a$Department of Mathematics and Statistics, Faculty of Science, Thaksin University,}
	\centerline{Phatthalung Campus,  Phatthalung 93110, Thailand} 
}

\medskip
{\small
 \centerline{$^b$Department of Mathematics, Faculty of Science,  Silpakorn University,}
   \centerline{Nakhon Pathom 73000,  Thailand}
} 

\medskip
{\small
 \centerline{$^c$Department of Mathematics and Computer Science, Faculty of Science,}
   \centerline{Chulalongkorn University, Bangkok 10330, Thailand}
}

\bigskip


\begin{abstract}
	Constacyclic and quasi-twisted Hermitian self-dual  codes over  finite  fields  are studied.  An algorithm for factorizing  $x^n-\lam$ over $\Fqq$ is given, where $\lambda$ is a unit in $\Fqq$.  Based on this factorization,   the dimensions of the Hermitian hulls of $\lam$-constacyclic codes of length $n$ over $\Fqq$ are determined.  The characterization and enumeration of  constacyclic  Hermitian self-dual (resp., complementary dual) codes of length $n$ over $\Fqq$ are given through their Hermitian hulls.  Subsequently, a new family of MDS  constacyclic Hermitian self-dual codes over $\Fqq$ is introduced. 
	
	As a generalization of constacyclic codes,      quasi-twisted Hermitian self-dual  codes are studied. Using the factorization of  $x^n-\lam$  and the Chinese Remainder Theorem,  quasi-twisted codes can be viewed as a product of linear codes of  shorter length some over  extension fields of $\Fqq$. Necessary and sufficient conditions for quasi-twisted codes to be Hermitian self-dual are given. The enumeration of such self-dual codes is determined as well.
\end{abstract}

\section{Introduction}
Quasi-twisted (QT) codes,  introduced  in \cite{Berlekamp}, 
play an important role in coding theory since they
contain remarkable classes of codes such as    quasi-cyclic (QC) codes, constacyclic codes, and cyclic codes. In \cite{Chepyzhov}, \cite{Kasami} and \cite{Ling2003},  it has been shown that QT and QC codes meet a modified version of the Gilbert-Vashamov bound.  Various codes with good  parameters and some optimal codes over finite fields 
have been  obtained from the  classes of QT and QC codes  (see \cite{Daskalov2003}, \cite{Aydin2007}, \cite{Chen2008} and  \cite{Aydin2001}).   Moreover, there is a link between QC codes and convolution codes in \cite{Gulliver1998} and \cite{Solomon1979}.

Constacyclic codes  are an important subclass of   QT codes due to their nice algebraic structures and various applications  in engineering 
\cite{Bakshi2012}, \cite{GHM2007} and \cite{Chen2012}.  Such codes are optimal in some  cases (see,  \cite{Yang13}, \cite{Chen}, \cite{La}, \cite{GHM2007} and  \cite{Kai2014}). These motivate the study of constacyclic codes     in \cite{Jia2011}, \cite{Bakshi2012}, \cite{Yang13},  \cite{Chen2012}, \cite{OO2009} and  \cite{LLB2015}.

Self-dual codes are another interesting class of codes due to their   fascinating links to other objects and their wide applications \cite{NRS2006} and \cite{Sloane}. 
Both Euclidean and Hermitian self-dual codes are also closely related to quantum stabilizer codes \cite{KKKS2006}. 
In \cite{LingI},  
\cite{LingIII}, \cite{LingIV} and  \cite{Jia}, QT and QC codes have been decomposed into a product of  linear codes of shorter length and  the Euclidean duals of such codes   have been determined via this decomposition. 
Consequently, the characterization  of   QT and QC Euclidean self-dual codes have been   given. In some cases, the enumeration of such codes has been established as well.

To the best of our knowledge,  only few  works have been done on Hermitian duals of constacyclic and QT codes. In \cite{Yang13}, a characterization of  Hermitian duals of constacyclic Hermitian self-dual  codes has been established but not an enumeration.  
It is therefore of natural interest to characterize and enumerate  constacyclic  and QT codes with Hermitian self-duality. 


Our goal is to study constacyclic and QT codes and their duals with respect to the Hermitian inner product which are defined over a finite field whose cardinality is square. Throughout the paper, we are therefore assume that the cardinality of a field is square and the notation $\Fqq$ will be used.

For a nonzero $\lam \in \Fqq$, let $o_{q^2}(\lam)$ denote the order of   $\lam$ in the multiplicative group $\F_{q^2}^\times:=\F_{q^2}\smallsetminus \{0\}$.  In \cite[Proposition 2.3]{Yang13}, it has been shown  that 
the Hermitian dual of a $\lam$-constacyclic code is also $\lam$-constacyclic if and only if $o_{q^2}(\lam)|(q+1)$. Later, in Proposition \ref{prop20}, we show that the Hermitian dual of a $(\lam, \ell)$-QT code over $\Fqq$ is again $(\lam, \ell)$-QT  if and only if $o_{q^2}(\lam) |(q+1)$. To study constacyclic and QT Hermitian self-dual codes, it suffices to restrict the study to the case where $o_{q^2}(\lam)|(q+1)$.    For $\lambda\in\{1,-1\}$ (or equivalently, $o_{q^2}(\lam)\in\{1,2\}$), $\lambda$-constacyclic  Hermitian self-dual codes have been studied in \cite{Sangwisut}.   In this paper, we give  the characterization and enumeration of $\lam$-constacyclic  Hermitian self-dual codes of any length $n$ and over $\Fqq$ for every nonzero $\lam\in\Fqq$ such that    $o_{q^2}(\lam)|(q+1)$. Subsequently, the characterization and enumeration of QT  Hermitian self-dual codes of  length $n\ell$ over $\Fqq$ are given in the case where $\gcd(q,n)=1$.

The paper is organized as follows. In Section 2, some preliminary concepts and proofs of some basic results are discussed.  An algorithm for  explicit factorization of $x^n-\lam$ over $\Fqq$ which is key to study constacyclic and QT codes is given in Section 3.
In Section 4,  the characterization of the Hermitian hulls of constacyclic codes of any length $n$ over $\Fqq$ is given.  Subsequently, necessary and sufficient conditions for constacyclic codes of length  $n$ over $\Fqq$ to be Hermitian self-dual (resp., Hermitian complementary dual) are determined together with the number of such codes.  A new family of MDS constacyclic Hermitian self-dual  codes over $\Fqq$ is introduced in Section 5. 
The decomposition for quasi-cyclic codes is generalized  to the case of    quasi-twisted codes  in Section 6.  The number of  $(\lam, \ell)$-QT Hermitian self-dual codes of length $n\ell$ over $\Fqq$ is also determined
\section{Preliminaries}
In this section, we  recall some  basic properties of   codes and polynomials over finite fields.

Let $\Fqq$ denote a finite field of order $q^2$. For a positive integer $n$, denote by $\Fqq^n$ the vector space of all vectors of length $n$  over $\Fqq$. A \textit{linear code} $C$ of length $n$ and dimension $k$ over $\Fqq$ is a $k$-dimensional subspace of $\Fqq^n$. 
A linear code $C$ over $\Fqq$ is said to have parameters  $[n,k,d]$ if  $C$ is of  length $n$, dimension $k$, and minimum  Hamming distance ${d}=\min\{\omega(\textbf{c})\mid \textbf{0}\neq \textbf{c}\in C \}$, where $\omega(\textbf{c})$ denotes the Hamming weight of  $\textbf{c}$. The parameters of every  $[n,k,d]$ linear code    satisfy the Singleton bound 
\[k\leq n-d+1.\]
An  $[n,k,{d}]$   linear code  over $\Fqq$ is said to be  a \textit{maximum distance separable (MDS) code} if  $k=n-{d}+1$.

For a linear code $C$ over $\Fqq$,   the {Euclidean dual} $C^{\perp_E}$ of   $C$ is defined under the  {\em Euclidean inner product}  \[\left\langle \textbf{a},\textbf{b}\right\rangle_E:=\sum_{i=0}^{n-1}a_ib_i, \]
where $\textbf{a}=(a_0,\ldots,a_{n-1}), \textbf{b}=(b_0,\ldots,b_{n-1})\in \Fqq^n$.
A code $C$ is said to be {\em Euclidean self-dual}  if $C=C^{\perp_E}$.


The {\em Hermitian dual} $C^{\perp_H}$ of $C$  is defined under  the {\em Hermitian inner product}
\[\left\langle \textbf{a},\textbf{b}\right\rangle_H:=\sum_{i=0}^{n-1}a_ib_i^q,\] where $\textbf{a}=(a_0,\ldots,a_{n-1}),$ $\textbf{b}=(b_0,\ldots,b_{n-1})\in \Fqq^n$.
The \textit{Hermitian hull} of $C$   is defined to be $Hull_H(C)=C\cap C^{\perp_H}$.    A linear code  $C$  is said to be \textit{Hermitian self-dual} 
(resp., \textit{Hermitian complementary dual}) if    $C=Hull_H(C)=C^{\perp_H}$   (resp., $Hull_H(C)=\{0\}$).  The Euclidean hull  of a linear code  $C$   is defined in the same fasion and studied in \cite{Sangwisut}.

\subsection{Constacyclic Codes}
Given a nonzero $\lambda\in\Fqq$, a linear code $C$ of length $n$ over $\Fqq$ is said to be \textit{constacyclic}, or specifically, \textit{$\lambda$-constacyclic} 
if for each $(c_0, c_1,\ldots, c_{n-1})\in C$, the vector
$(\lambda c_{n-1}, c_0,\ldots, c_{n-2})$
is again a codeword in $C$.
A $\lam$-constacyclic code is called {\em cyclic} and {\em negacyclic} if $\lam=1$ and $\lam=-1$, respectively. It is well known (see, for example, \cite{Yang13}) that every  $\lam$-constacyclic code $C$ of length $n$ over $\Fqq$ can be  identified with  an ideal in   $\Fqq[x]/\langle x^n-\lam\rangle$   generated by a unique monic divisor of $x^n-\lam$. Such a  polynomial is  called the {\em generator polynomial} of $C$.

Given a polynomial $f(x)=a_0+a_1x+\ldots+a_kx^k\in \Fqq[x]$ with  nonzeros $a_0$ and $a_k$, denote by  $f^\dagger(x):=a_0^{-q}\sum_{i=0}^ka_i^qx^{k-i}$  the \textit{conjugate-reciprocal polynomial} of $f(x)$. The polynomial $f(x)$ is said to be \textit{self-conjugate-reciprocal} if $f(x)=f^\dagger(x)$. Otherwise, $f(x)$ and $f^\dagger(x)$ are called a \textit{conjugate-reciprocal polynomial pair}.

Let $g(x)$ be the generator polynomial of a $\lambda$-constacyclic code $C$ of length $n$ over $\Fqq$ and let $h(x)=\frac{x^n-\lambda}{g(x)}$. Then $h^\dagger(x)$ is a monic divisor of $x^n-\lambda$ and it is the generator polynomial of $C^{\perp_H}$ (see \cite[Lemma 2.1]{Yang13}). Therefore, $C$ is Hermitian self-dual if and only if $g(x)=h^\dagger(x)$.
By \cite[Theorem 1]{Sangwisut}, $Hull_H(C)$ is generated by ${\rm lcm}(g(x),h^\dagger(x))$.  
\subsection{Quasi-Twisted Codes}\label{qt}

View a codeword in  a linear code $C$ of length $n\ell $ over $\Fqq$ as an $n\times \ell$ matrix over   $\Fqq$.
Given a nonzero $\lam\in\Fqq$, a linear code $C$ of length $n\ell $ over $\Fqq$ is said to be \textit{$(\lam,\ell)$-quasi-twisted ($(\lam,\ell)$-QT) of length $n\ell$ over $\Fqq$} if for each

\begin{align*}
\textbf{c}=\left[
\begin{array}{cccc}
c_{00}&c_{01}&\dots&c_{0,\ell-1}\\
c_{10}&c_{11}&\ldots&c_{1,\ell-1}\\
\vdots&\vdots&\ddots&\vdots\\
c_{n-1,0}&c_{n-1,1}&\ldots&c_{n-1,\ell-1}
\end{array}
\right]\in C, \end{align*}
the vector 
\begin{align*} 
\textbf{c}'=\left[
\begin{array}{cccc}
\lambda c_{n-1,0}&\lambda c_{n-1,1}&\dots&\lambda c_{n-1,\ell-1}\\
c_{00}&c_{01}&\ldots&c_{0,\ell-1}\\
\vdots&\vdots&\ddots&\vdots\\
c_{n-2,0}&c_{n-2,1}&\ldots&c_{n-2,\ell-1}
\end{array}
\right]
\end{align*}
is again a codeword in $C$. 
We define an action $T_{\lam, \ell}$ on the codewords as $T_{\lam, \ell}(\textbf{c})=\textbf{c}'$. Then every  $(\lam, \ell)$-QT code is invariant as a subspace under the action $T_{\lam,\ell}$.

Let $R:={\Fqq[x]}/{\left\langle x^n-\lam\right\rangle}$. Define a map $\psi:\Fqq^{n\ell }\rightarrow R^\ell$ by
\begin{align}\label{psi}\psi(\textbf{c})=\begin{bmatrix}c_0(x)\\c_1(x)\\\vdots\\c_{\ell-1}(x) \end{bmatrix}=\begin{bmatrix}c_{00}+c_{10}x+~~\ldots~~+c_{n-1,0}x^{n-1}\\c_{01}+c_{11}x+~~\ldots~~+c_{n-1,1}x^{n-1}\\\vdots\\c_{0,\ell-1}+c_{1,\ell-1}x+\ldots+c_{n-1,\ell-1}x^{n-1}\\  \end{bmatrix}.
\end{align}
Then the next lemma follows.
\begin{lem}\label{eqipp}
	The map $\psi$ induces a one-to-one correspondence between the QT-codes of length $n\ell $ over $\Fqq$ and the $R$-submodules of $R^\ell$.  
\end{lem}


\section{The Factorization of $x^n-\lam$ in $\Fqq[x]$}
In this section, we give an algorithm for the factorization of  $x^n-\lam$ in $\Fqq[x]$  which is key to study both the structures of $\lambda$-constacyclic and $(\lam, \ell)$-QT codes.  

Let $\lam$ be a nonzero element in $ \Fqq$ such that $o_{q^2}(\lam)=r$ and let $n$ be a positive integer   written in the form of $n={n^\prime}p^\nu$, where $p={\rm char}(\Fqq),~ p\not\mid n^\prime $ and $\nu\geq 0$. 

Since the map $a\mapsto a^{p^{\nu}}$ on $\Fqq$ is a power of the Frobenious automorphism of $\Fqq$ over $\F_p$, there is a unique ${\Lambda}\in\Fqq$ such that ${\Lambda}^{p^{\nu}}=\lambda$. Then  
\begin{align}\label{xn-lam}
x^n-\lam=\left(x^{{n^\prime}}-{\Lambda}\right)^{p^\nu}.
\end{align} 
Since an  automorphism is order preserving, we have $o_{q^2}({\Lambda})=o_{q^2}(\lam)=r$. Therefore, it is sufficient to  focus on the factorization of $x^{{n^\prime}}-{\Lambda}$.

Let $k$ be the smallest integer such that ${(n^\prime}r)|(q^{2k}-1)$. Then, there exists  a primitive ${n^\prime}r$th root of  $\xi$ in $\F_{q^{2k}}$ such that $\xi^{{n^\prime}}={\Lambda}$, and hence,  
\begin{align}\label{xn-xi}
x^{n^\prime}-\Lambda=x^{n^\prime}-\xi^{n^\prime}.
\end{align}
It is not dificult to see that $x^{n^\prime}-\xi^{n^\prime}$ divides $x^{{n^\prime}r}-1 $. Since 
$x^{{n^\prime}r}-1=\displaystyle\prod_{j\mid n^\prime r}Q_j(x)$, where $$Q_j(x):=\prod_{z\in \Z_j^\times} \left(x-\xi^{\left(\frac{{n^\prime}r}{j}\right)z}\right)$$
is the \textit{$j$th cyclotomic polynomial} over $\Fqq$ (see \cite{Jia2011}) , we have   
\begin{align}\label{bdcd}
x^{n^\prime}-\xi^{n^\prime} &=\gcd\left(\prod_{j\mid n^\prime r}Q_j(x), x^{{n^\prime}}-\xi^{{n^\prime}}\right)   =\prod_{j\mid n^\prime r}\gcd\left(Q_j(x), x^{{n^\prime}}-\xi^{{n^\prime}}\right).
\end{align}
Hence, for each divisor $j$ of $n^\prime r$,~~$ \xi^{\left(\frac{{n^\prime}r}{j}\right)z}$ is a root of $x^{n^\prime}-\xi^{n^\prime}$ if and only if $\xi^{{n^\prime}\left(\frac{{n^\prime}r}{j}\right)z}=\xi^{{n^\prime}}\,$, or equivalently, $\left(\frac{{n^\prime}r}{j}\right)z\equiv 1\,{\rm mod}\, r$.
The set of elements in $\Z_{{n^\prime}r}$ satisfying the preceeding conditions is denoted by
\begin{align*}
S_j:=\left\{\left(\frac{{n^\prime}r}{j}\right)z\in\Z_{{n^\prime}r}\,\left\vert\, z\in \Z_j^\times,\; {\left(\frac{{n^\prime}r}{j}\right)z}\equiv {1}\,{\rm mod}\, r\right.\right\}.
\end{align*}
In other words, $S_j$ is the set of all $s$'s such that $\xi^s$ is a root of 
\[\gcd(Q_j(x), x^{{n^\prime}}-\xi^{{n^\prime}})=\displaystyle\prod_{s\in S_j}(x-\xi^{s}).\] It follows that
\begin{align*}
\deg \gcd(Q_j(x), x^{{n^\prime}}-\xi^{{n^\prime}})=|S_j|.
\end{align*}
For each $j\mid n^\prime r$,   necessary and sufficient conditions for $S_j$ to be nonempty 
are given in the following proposition.
\begin{prop}\label{nonempt} Let $j$ be a positive divisor of $n^\prime r$. Then
	$S_j\neq\emptyset$ if and only if \[\gcd\left(\frac{{n^\prime }r}{j},r\right)=1.\]
\end{prop}
\begin{proof}
	Assume that $S_j\neq \emptyset$. Then there exists $\left(\frac{{n^\prime }r}{j}\right)z\in S_j$. Then $\left(\frac{{n^\prime }r}{j}\right)z-rm=1$ for some $m\in\N$. It follows that $\gcd\left(\frac{{n^\prime }r}{j},r\right)=1$.

	Conversely, assume that  $\gcd\left(\frac{{n^\prime }r}{j},r\right)=1$. Then there exists $w_1\in\Z_r^\times$ such that $\frac{{n^\prime }r}{j} w_1\equiv 1 \,{\rm mod}\, r$. Observe that $(rm+w_1)\left(\frac{{n^\prime }r}{j}\right)\equiv 1 \,{\rm mod}\, r$ for all $m\in\Z^+$. By Dirichlet's theorem on arithmetic progressions (see \cite{Serre1973}), there exist infinitely many primes of the form $rm+w_1$. Let $m_1\in\Z^+$ be such that $rm_1+w_1$ is prime and $rm_1+w_1>j$. Hence, we obtain $w=(rm_1+w_1)\,{\rm mod}\,j$ such that $w\in\Z_j^\times$ and $w\left(\frac{{n^\prime }r}{j}\right)\equiv 1 \,{\rm mod}\, r$. Therefore, $S_j\neq \emptyset$ as desired.
\end{proof}

From now on, we focus only on the  positive divisors $j$ of $n^\prime r$ such that $S_j\neq\emptyset$, or equivalently, $\gcd\left(\frac{{n^\prime }r}{j},r\right)=1$. The cardinality of  $ S_j$ is determined in the following lemma.

\begin{lem}\label{lem5.8}
	Let $j$ be a positive divisor of ${n^\prime }r$  such that $\gcd\left(\frac{{n^\prime }r}{j}, r\right)=1$. Then $|S_j|=\frac {\phi(j)}{\phi(r)}$, where $\phi$ is the Euler's totient function.
\end{lem}
\begin{proof}
	Let $H_j$ be defined by
	$H_j:=\left\{h\in \Z_j^\times \mid h\equiv 1\,{\rm mod}\, r \right\}.$
	We divide the proof into two steps. First, we show that $|S_j|=|H_j|$. Then we determine  $|H_j|$.
	
	By Proposition \ref{nonempt},  we have  $S_j\neq\emptyset$. Let $\left(\frac{{n^\prime }r}{j}\right)w\in S_j$, where $w\in \Z_j^\times$. Then there exists  $w'\in \Z_j^\times$  such that $ww'\equiv 1\,{\rm mod}\, j$.
	
	Let  $\Phi:S_j\rightarrow H_j$ be defined  by $\Phi\left(\left(\frac{n^\prime r}{j}\right)z\right)=w'z$. Since $w'z\equiv w'w\left(\frac{n^\prime r}{j}\right)z\equiv\left(\frac{n^\prime r}{j}\right)z\equiv 1\,{\rm mod}\, r$,  we have $w'z\in H_j$. Let $\left(\frac{n^\prime r}{j}\right)z_1=\left(\frac{n^\prime r}{j}\right)z_2$ in $S_j$.  Then  $\left(\frac{n^\prime r}{j}\right)(z_1-z_2)\equiv 0\,{\rm mod}\, n^\prime r$.  Since $j=\frac{n^\prime r}{n^\prime r/j}$,  we have  $z_1\equiv z_2\,{\rm mod}\, j$,  and hence,  $w'z_1\equiv w'z_2\,{\rm mod}\, j$. Therefore,  $\Phi$ is well-defined.
	
	Let  $z_1, z_2\in \Z_j^\times$ be such that $\Phi\left(\left(\frac{n^\prime r}{j}\right)z_1 \right)=\Phi\left(\left(\frac{n^\prime r}{j}\right)z_2 \right)$.
	Then $w'z_1=w'z_2$ in $\Z_j^\times$, {\em i.e.}, $w'z_1\equiv w' z_2\,{\rm mod}\,j$. 
	Hence, we have  $z_1\equiv z_2\,{\rm mod}\,j$. 
	It follows that $\left(\frac{n^\prime r}{j}\right)z_1\equiv \left(\frac{n^\prime r}{j}\right)z_2\,{\rm mod}\,n^\prime r$, and hence,  $\left(\frac{n^\prime r}{j}\right)z_1=\left(\frac{n^\prime r}{j}\right)z_2$ in $S_j$.  Therefore, $\Phi$ is injective.
	
	For each  $h\in H_j$, we have   $\left(\frac{n^\prime r}{j}\right)wh$ in $S_j$ and  $\Phi\left(\left(\frac{n^\prime r}{j}\right)wh\right)=w'wh\equiv h \,{\rm mod}\, j$. Then  $\Phi$ is surjective, and hence, it is a bijection.  Therefore,  $|S_j|=|H_j|$.


	By the Fundamental Theorem of Arithmetic, we have $j=p_1^{a_1}\ldots p_t^{a_t}$, where $p_1<p_2<\cdots<p_t$ are primes and $a_i$ is a positive integer. Since $\gcd(\frac{n^\prime r}{j}, r)=1$, we have $r| j$. Hence, we can write $r=p_1^{b_1}\ldots p_t^{b_t}$, where $b_i$ is non-negative integer and $b_i\leq a_i$ for all $1\leq i\leq t$. By the Chinese Remainder Theorem, 
	$$\Z_j^\times\cong \Z_{p_1^{a_1}}^\times\times\Z_{p_2^{a_2}}^\times\times\cdots\times \Z_{p_t^{a_t}}^\times$$
	and  each element in $H_j$ corresponds to
	$(z_1,\ldots,z_t)$ in $H_{p_1^{a_1}}\times\cdots\times H_{p_t^{a_t}}$.  	Therefore,
	\begin{align}\label{hp}
	|H_j|=|H_{p_1^{a_1}}|\cdot |H_{p_2^{a_2}}|\cdots|H_{p_t^{a_t}}|.
	\end{align}
	Note  that, for each $1\leq i\leq t$,  \[H_{p^{a_i}}=\left\{z\in\Z_{p^{a_i}}^\times\mid z\equiv 1 \,{\rm mod}\, p^{b_i} \right\} =\left\{1, 1+p^{b_i}, 1+2p^{b_i},\ldots,1+(p^{a_i-b_i}-1)p^b \right\}.\] Then  $|H_{p^{a_i}}|=p^{a_i-b_i}=\frac{p^{a_i}}{p^{b_i}}=\frac{\phi(p^{a_i})}{\phi(p^{b_i})}$.
	
	From (\ref{hp}), we conclude that 
	\begin{align*}
	|H_j|=\frac{\phi(p_1^{a_1})}{\phi(p_1^{b_1})}\cdot\frac{\phi(p_2^{a_2})}{\phi(p_2^{b_2})}\cdots \frac{\phi(p_t^{a_t})}{\phi(p_t^{b_t})}=\frac{\phi(j)}{\phi(r)}
	\end{align*}
	as desired.
\end{proof}

Therefore, for each  divisor $j$ of ${n^\prime }r$  with $\gcd\left(\frac{{n^\prime }r}{j}, r\right)=1$, we have 
$$\deg\gcd(Q_j(x), x^{{n^\prime}}-\xi^{{n^\prime}})=\frac {\phi(j)}{\phi(r)}.$$

Let $\pi$ be a map  defined on the pair $(j,q^2)$, where $i$ is a positive integer, by
$$\pi(j,q^2):=\begin{cases}
0 ~~&\text{~if~} j|(q^{2k}+q) \text{~for some~} k\geq 0,\\
1 &\text{~otherwise.~}
\end{cases}$$

For each positive integer $j$ such that $\gcd(j,q)=1$, \textit{the order of $q^2$ in the multiplicative group $\Z_j^\times$} is denoted by $ord_j(q^2)$.
The following lemma can be obtained by replacing $q$ with  $q^2$  in the proofs of \cite[Lemma 3 and Lemma 19]{Sangwisut}.
\begin{lem} 
	\label{numgb2}
	Let $j$ be a positive integer and let $\Fqq$ be a finite field with $\gcd(j,q)=1$. 
	The $j${\rm th} cyclotomic polynomial $Q_j(x)$ factors into $\frac{\phi(j)}{ord_j(q^2)}$ distinct monic irreducible polynomials over $\Fqq$ of the same degree $ord_j(q^2)$, where $\phi$ is the Euler's totient function.
	
	If $\pi(j, q^2)=0$, then all the irreducible polynomials in the factorization of $Q_j(x)$ are self-conjugate-reciprocal. 
	Otherwise, they form conjugate-reciprocal polynomial pairs.
\end{lem}

By Lemma \ref{numgb2}, $\gcd(Q_j(x), x^{{n^\prime}}-\xi^{{n^\prime}})$ can be factored into $\frac{\phi(j)}{\phi(r)ord_j(q^2)}$ distinct monic irreducible polynomials over $\Fqq$ of the same degree $ord_j(q^2)$.
In addition,
\begin{align}\label{gcd}
\gcd(Q_j(x), x^{{n^\prime}}-\xi^{{n^\prime}})=\begin{cases}
\displaystyle\prod_{i=1}^{\gam(j)}g_{ij}(x) &\text{if\;} \pi(j, q^2)=0,\\
\displaystyle\prod_{i=1}^{\be(j)} f_{ij}(x)f^\dagger_{ij}(x) &\text{otherwise},
\end{cases}
\end{align}
where
\begin{align}\label{gamm} 
\gam(j):=\frac {\phi(j)}{\phi(r)ord_j(q^2)},
\end{align}

\begin{align}\label{bet}
\be(j):=\frac {\phi(j)}{2\phi(r)ord_j(q^2)},
\end{align}
$f_{ij}(x)$ and $f^\dagger_{ij}(x)$ are a monic irreducible conjugate-reciprocal polynomial pair, 	 and ~$g_{ij}(x)$~ is a monic irreducible self-conjugate-reciprocal polynomial.   

By \eqref{xn-lam}-\eqref{bdcd},  and \eqref{gcd}, it can be  concluded that
\begin{align} \label{fac}
x^n-\lam=&\left(x^{{n^\prime}}-{\Lambda}\right)^{p^{\nu}}=\left(x^{{n^\prime}}-\xi^{{n^\prime}}\right)^{p^{\nu}} \nonumber \\\nonumber
=& \left(\prod_{j\mid {n^\prime}r,\; \gcd\left(\frac{{n^\prime}r}{j} ,r\right)=1} \gcd(Q_j(x), x^{{n^\prime}}-\xi^{{n^\prime}})\right)^{p^{\nu}}\\
=& \prod_{\substack{j\mid {n^\prime}r,\; \gcd\left(\frac{{n^\prime}r}{j} ,r\right)=1 \\\pi(j, q^2)=0}}\prod_{i=1}^{\gam(j)}\left(g_{ij}(x)\right)^{p^{\nu}}\prod_{\substack{j\mid {n^\prime}r,\; \gcd\left(\frac{{n^\prime}r}{j} ,r\right)=1 \\\pi(j, q^2)=1}}\prod_{i=1}^{\be(j)}\left(f_{ij}(x)\right)^{p^{\nu}}\left(f^\dagger_{ij}(x)\right)^{p^{\nu}}.
\end{align}
For simplicity, let
\begin{align} \label{omega}
\Omega=\left\{j\mid {n^\prime}r ~\middle|~  \gcd\left(\frac{{n^\prime}r}{j},r\right)=1\text{~and~}  \pi(j, q^2)=0\right\}
\end{align}
and
\begin{align} \label{omegaprime}
\Omega'=\left\{j\mid {n^\prime}r ~\middle|~  \gcd\left(\frac{{n^\prime}r}{j},r\right)=1\text{~and~}  \pi(j, q^2)=1\right\}.
\end{align}
Then (\ref{fac}) becomes
\begin{align}
\label{xn-1}
x^n-\lam=& \prod_{j\in\Omega}\prod_{i=1}^{\gam(j)}\left(g_{ij}(x)\right)^{p^{\nu}}\prod_{j\in\Omega'}\prod_{i=1}^{\be(j)}\left(f_{ij}(x)\right)^{p^{\nu}}\left(f^\dagger_{ij}(x)\right)^{p^{\nu}}.
\end{align}
Let $\mathtt{s}$ and  $\mathtt{t}$ denote  the number of monic irreducible self-conjugate-reciprocal polynomials and  the number of monic irredcible conjugate-reciprocal polynomial pairs in the factorization of $x^{{n^\prime}}-{\Lambda}$, respectively.  Then
\begin{align}
\label{ss}\mathtt{s}=\sum_{j\in\Omega}\gam(j) 
\end{align}
and 
\begin{align}\label{tt}
\mathtt{t}=\sum_{j\in\Omega'}\be(j),
\end{align}
where $\gamma$ and $\beta$ are defined in (\ref{gamm})  and (\ref{bet}), respectively.

\begin{ex}
	Consider $\F_4=\left\{0, 1, \al, \al^2=\al+1\right\}$. Let $n=5$. Then $o_4(\al)=3$ and \[\left\{j\mid j \text{ is a divisor of } 15 \text{ and } \gcd\left(\frac{15}{j},3\right)=1 \right\}=\left\{3, 15\right\}.\]  Since  $\pi(3,4)=0,~ \pi(15,4)=1$ and $ord_3(4)=1 $ and $ord_{15}(4)=2$, we have $\gam(3)=\frac{\phi(3)}{\phi(3)ord_3(4)}=1=\be(4)=\frac{\phi(15)}{2\phi(3)ord_{15}(4)}$. Therefore, by (\ref{xn-1})-(\ref{tt}),    the factors of $x^5-\al$  contains $\gamma(3)=1$ irreducible self-conjugate-reciprocal polynomial of dergree $ord_{3}(4)=1$ and $\gamma(5)=1$  irreducible conjugate-reciprocal polynomial pair of dergree  $ord_{15}(4)=2$.
\end{ex}

From the discussion above, we can determine the degrees and the number of self-conjugate-reciprocal polynomials and conjugate-reciprocal polynomial pairs in the factorization of $x^{n^\prime}-\Lambda$ in \eqref{xn-1}.  However, we are not yet able to determine the explicit irreducible factors of $x^{n^\prime}-\Lambda$.  The following algorithm gives the  explicit  factors of $x^{n^\prime}-\Lambda$.

A \textit{$q^2$-cyclotomic coset modulo  ${n^\prime}r$  containing  $a$}, denoted by $S_{q^2}(a)$, is defined to be the set
$$S_{q^2}(a):=\left\{q^{2i}\cdot a \;{\rm mod}\;  {n^\prime}r\mid i=0,1,\ldots \right\}.$$
Since  $\gcd(Q_j(x), x^{{n^\prime}}-\xi^{{n^\prime}})$ can be factored as a product of irreducible  polynomials in $\Fqq[x]$,  $S_j$ is a union of some $q^2$-cyclotomic cosets modulo ${n^\prime}r$.  Therefore, we conclude the following algorithm.
\begin{center}
	\textbf{Algorithm} 
\end{center}
\begin{enumerate}
	\item For each $j|{n^\prime}r$ such that $\gcd\left(\frac{{n^\prime}r}{j},r\right)=1$, find the set $S_j$.
	\item Partition $S_j$ into $q^2$-cyclotomic cosets modulo ${n^\prime}r$.
	\item Determine $\pi(j, q^2)$. 
	\begin{enumerate}[(3.1)]
		\item If $\pi(j, q^2)=1$, then denote by ${T}_j$ a set of $q^2$-cyclotomic cosets of $S_j$ such that $S_{q^2}(a)\in{T}_j$ if and only if $S_{q^2}(-qa)\not\in{T}_j$. Let $\mathscr{T}_j$ denote a set of representative of $q^2$-cyclotomic cosets in each $q^2$-cyclotomic cosets in $T_j$.
		\item If $\pi(j, q^2)=0$, let $\mathscr{S}_j$ denote a set of representative of $q^2$-cyclotomic cosets in $S_j$.
	\end{enumerate}
	\item We have 
	\[x^{{n^\prime}}-{\Lambda}=\prod_{a\in\mathscr{S}_j} m_{\xi^a}(x)\prod_{b\in \mathscr{T}_j}m_{\xi^b}(x)m_{\xi^b}^\dagger(x).\]
\end{enumerate}

The following example illustrates an application of the algorithm.
\begin{ex}\label{ex1}
	Let $\F_4=\left\{0, 1, \al, \al^2=\al+1 \right\}$. Then $o_{4}(\al)=3$. To factor $x^5-\al$ over $\F_4$, let $\xi$ be a primitive 15th root of unity in $\F_{16}$ such that $\al=\xi^{5}$.  Note that all $j|5\cdot 3$ with $\gcd\left(\frac{15}{j},3\right)=1$ are $3$ and $15$. Since $\pi(3,4)=0$ and  $\pi(15,4)=1$,  we have
	$
	S_3=\left\{5z\,\left\vert\, z\in\Z_3^\times, 5z\equiv 1\,{\rm mod}\, 3 \right.\right\}=\{10\}=\mathscr{S}_3
	$
	and 
	$
	S_{15}=\left\{z\,\left\vert\, z\in\Z_{15}^\times, z\equiv 1\,{\rm mod}\, 3 \right.\right\}=\{1, 4, 7, 13\}.
	$
	Partitioning $S_{15}$ into $4$-cyclotomic coset modulo $15$, we have $S_{15}=\{1, 4\}\cup \{7, 13\}$. Then ${T}_{15}=\left\{\left\{1, 4\right\}\right\}$ and $\mathscr{T}_{15}=\left\{ 1\right\}$. 
	Therefore, $x^5-\al$ can be written in term of equation \eqref{xn-1} as  
	\begin{align*}
	x^5-\al&=m_{\xi^{10}}(x)\left( m_{\xi}(x)m_{\xi}(x)^\dagger\right)=(x+\al^2)(x^2+x+\al)(x^2+\al x+\al).
	\end{align*}
\end{ex} 
\section{ Hermitian Hull of $\lam$-Constacyclic Codes} 
In this section,  the dimensions of the Hermitian hulls of constacyclic codes of length $n$ over $\Fqq$ are determined via the factorization of $x^n-\lambda$ given in Section 3.  
The number of  constacyclic Hermitian self-dual codes and the number of Hermitian complementary dual constacyclic codes of length $n$ over $\Fqq$ are  given as well.

\begin{thm}\label{thm24} Let $\mathbb{F}_{q^2}$ denote a finite field of order $q^2$ with characteristic $p$ and let $n=\overline{n}p^{\nu} $ with $p\not\mid \overline{n}$.
	Then the  dimensions of the Hermitian hulls of $\lam$-constacyclic codes of  length $n$  over $\Fqq$  are of the form
	\begin{align}\label{eq24}
	\sum_{j\in\Omega} ord_j(q^2) \cdot \mathsf{a}_j + \sum_{j\in\Omega'} ord_j(q^2)\cdot \mathsf{b}_j,
	\end{align}
	where $0\leq \mathsf{a}_{j}\leq\gam(j)\left\lfloor \frac{p^{\nu}}{2}\right\rfloor$ and $0\leq \mathsf{b}_{j}\leq \be(j)p^{\nu}$.
\end{thm}
\begin{proof}
	The theorem can be obtained using arguments similar to those in the proof of  \cite[Theorem 5]{Sangwisut} by replacing $\chi$ with $\pi$ and $q$ with $q^2$.
\end{proof}

Next theorem gives a characterization of $\lam$-constacyclic Hermitian self-dual codes in terms of $\Omega$ defined  in \eqref{omega}.
\begin{thm}\label{thm7}
	Let $\mathbb{F}_{q^2}$ denote a finite field of order $q^2$ with characteristic $p$ and let $n=\overline{n}p^{\nu} $ with $p\not\mid \overline{n}$.
	Let $x^n-\lam$ be factored as in \eqref{xn-1}. 
	Then there exists a   $\lambda$-constacyclic Hermitian self-dual code of length $n$ over $\mathbb{F}_{q^2}$ if and only if 
	\begin{enumerate}
		\item $\Omega=\emptyset$, or
		\item $\Omega\neq \emptyset$ ~~and~~ $p=2$.
	\end{enumerate}
	In this case, the generator polynomial of a code is of the form
	\begin{align}\label{genhsd}
	g(x)=\begin{cases}
	\displaystyle\prod_{j\in\Omega'}\prod_{i=1}^{\be(j)}\left(f_{ij}(x)\right)^{v_{ij}}\left({f^\dagger_{ij}(x)}\right)^{w_{ij}}&\text{if\;\;} \Omega=\emptyset, \\
	\displaystyle\prod_{j\in\Omega}\prod_{i=1}^{\gam(j)}\left({g_{ij}(x)}\right)^{2^{\nu-1}}\prod_{j\in\Omega'}\prod_{i=1}^{\be(j)}\left(f_{ij}(x)\right)^{v_{ij}}\left({f^\dagger_{ij}(x)}\right)^{w_{ij}}&\text{if\;\;} \Omega\neq \emptyset \text{~~and~~} p=2,
	\end{cases}
	\end{align}
	where $0\leq v_{ij}, w_{ij}\leq p^{\nu}$ and $v_{ij}+w_{ij}=p^{\nu}$.
\end{thm}

\begin{proof}
	Let $C$ be a $\lambda$-constacyclic code of length $n$ over $\Fq$ with the generator polynomial $g(x)$. Then, by (\ref{xn-1}), we have
	\begin{align}\label{genxn1-}
	g(x)= \prod_{j\in\Omega}\prod_{i=1}^{\gam(j)}{\left(g_{ij}(x)\right)}^{u_{ij}}\prod_{j\in\Omega'}\prod_{i=1}^{\be(j)}{\left(f_{ij}(x)\right)}^{v_{ij}}{\left(f^\dagger_{ij}(x)\right)}^{w_{ij}},
	\end{align}
	where $0\leq u_{ij}, v_{ij}, w_{ij}\leq p^{\nu}$. It follows that
	\[h(x):=\frac{x^{{n}}-\lam}{g(x)}=\prod_{j\in\Omega}\prod_{i=1}^{\gam(j)}{\left(g_{ij}(x)\right)}^{p^{\nu}-u_{ij}}\prod_{j\in\Omega'}\prod_{i=1}^{\be(j)}{\left(f_{ij}(x)\right)}^{p^{\nu}-v_{ij}}{\left(f^\dagger_{ij}(x)\right)}^{p^{\nu}-w_{ij}},\]
	and hence,
	\begin{align*}
	h^\dagger(x)=\frac{x^{{n}}-\lam}{g(x)}=\prod_{j\in\Omega}\prod_{i=1}^{\gam(j)}{\left(g_{ij}(x)\right)}^{p^{\nu}-u_{ij}}\prod_{j\in\Omega'}\prod_{i=1}^{\be(j)}{\left(f_{ij}(x)\right)}^{p^{\nu}-w_{ij}}{\left(f^\dagger_{ij}(x)\right)}^{p^{\nu}-v_{ij}}.
	\end{align*}
	
	Assume that   $C$ is Hermitian self-dual. Then $g(x)=h(x)^\dagger$.  By comparing the exponents, we have
	\begin{center}
		$u_{ij}=p^{\nu}-u_{ij}$ \;\;and\;\; $v_{ij}=p^{\nu}-w_{ij}$,\end{center}
	and hence, $2u_{ij}=p^{\nu}$  and $v_{ij}+w_{ij}=p^{\nu}$.  Since $2u_{ij}=p^{\nu}$,  we have    $p=2$ or $\Omega=\emptyset$.
	

	Conversely, assume that  $\Omega=\emptyset$, or
	$\Omega\neq \emptyset$ ~~and~~ $p=2$.  Let $g(x)$ be defined as in (\ref{genhsd}) and $h(x)=\displaystyle\frac{(x^n-\lambda)}{g(x)}$. It is not difficult to see that $g(x)=h^\dagger(x)$, and hence, a constacyclic code generated by $g(x)$ is Hermitian self-dual. 
\end{proof}

\begin{cor}\label{516} Let $\mathtt{t}$ be the number of monic irreducible conjugate-reciprocal polynomial pairs as in \eqref{tt}.
	The number of $\lam$-constacyclic Hermitian self-dual  codes of length $n$ over $\Fqq$ is
	\begin{align*}
	\begin{cases}
	(p^{\nu}+1)^{\mathtt{t}}&\text{if\;\;} \Omega=\emptyset, \\
	(2^{\nu}+1)^{\mathtt{t}}&\text{if\;\;}\Omega\neq \emptyset \text{~~and~~} p=2, \\
	\;\;\;\;\; 0 &\text{if\;\;}\Omega\neq \emptyset \text{~~and~~} p\neq 2.
	\end{cases}
	\end{align*}
	
	In particular, if $\Omega^\prime=\emptyset$ (or equivalently, $\pi({\overline{n}}r, q^2)=0$) and $p=2$, then there exists a unique $\lam$-constacyclic Hermitian self-dual  code. In this case, the generator polynomial is $$\prod_{j\in\Omega}\prod_{i=1}^{\gam(j)}{\left(g_{ij}(x)\right)}^{2^{\nu-1}}.$$
\end{cor}
\begin{proof}
	By Theorem \ref{thm7}, the number of generator polynomials of $\lam$-constacyclic   Hermitain self-dual codes of length $n$ over $\Fqq$ depends only on $v_{ij}$ and $w_{ij}$ such that $v_{ij}+w_{ij}=p^\nu$ where $0\leq v_{ij}, w_{ij}\leq p^\nu$. Then the number of   $\lam$-constacyclic Hermitain self-dual codes of length $n$ over $\Fqq$ is $(p^\nu+1)^t$.
	
	Since the number of generator polynomials of $\lambda$-constacyclic Hermitian self-dual codes of length $n$ over $\Fqq$ depends only on $v_{ij}$ and $w_{ij}$, a unique $\lambda$-constacylic Hermitian self-dual code occurs if $\Omega^\prime=\emptyset$ and $p=2$. It is not difficult to see that $\Omega^\prime=\emptyset$ is equivalent to $\pi({n^\prime}r, q^2)=0$. Therefore, the  generator polynomial of the code is \[\prod_{j\in\Omega}\prod_{i=1}^{\gam(j)}{g_{ij}(x)}^{2^{\nu-1}}.\]
\end{proof}

Necessary and sufficient conditions for constacyclic Hermitian complementary dual codes are given as follows.
\begin{thm}\label{complementary dual} Let $\mathbb{F}_{q^2}$ denote a finite field of order $q^2$ with characteristic $p$ and let $n=\overline{n}p^{\nu} $ with $p\not\mid \overline{n}$.
	Let $C$ be a $\lam$-constacyclic code of length $n$ over $\Fqq$.  Then $C$ is Hermitian complementary dual if and only if its generator polynomial is of the form
	\begin{align*}
	\prod_{j\in\Omega}\prod_{i=1}^{\gam(j)}\left({g_{ij}(x)}\right)^{u_{ij}}\prod_{j\in\Omega'}\prod_{i=1}^{\be(j)}\left(f_{ij}(x)\right)^{v_{ij}}\left(f^\dagger_{ij}(x)\right)^{w_{ij}},
	\end{align*}
	where $u_{ij}\in \left\{0, p^{\nu}\right\}$, and  $(v_{ij}, w_{ij})\in\left\{(0, 0), (p^{\nu},p^{\nu})\right\}$.
\end{thm}
\begin{proof} In the proof of Theorem \ref{thm7}, we have $\lam$-constacyclic codes $C$ and $C^\perp$  of length $n$ over $\Fqq$ generated by $g(x)$  and $h^\dagger(x)$ respectively. Hence, $C$ is a $\lam$-constacyclic Hermitian complementary dual code if and only if ${\rm lcm}(g(x),h^\dagger(x))=x^n-\lam$ or, equivalently $\max\{u_{ij}, p^\nu-u_{ij} \}=p^\nu, \max\{v_{ij}, p^\nu-w_{ij} \}=p^\nu$ and $\max\{w_{ij}, p^\nu-v_{ij} \}=p^\nu$. Thus, $u_{ij}\in \left\{0, p^{\nu}\right\}$, and  $(v_{ij}, w_{ij})\in \{(0, 0), (p^{\nu},p^{\nu})\}$.
\end{proof}

\begin{cor}\label{520} The number of $\lam$-constacyclic Hermitian complementary dual  codes of length $n$ over $\Fqq$ is $$2^{\mathtt{s}+\mathtt{t}},$$ 
	where  $\mathtt{s}$ is the number of monic irreducible self-conjugate-reciprocal polynomials and  $\mathtt{t}$ is the number of monic irreducible conjugate-reciprocal polynomial pairs as in \eqref{tt}.
\end{cor}
\begin{proof}
	From the proof of Theorem \ref{complementary dual}, the number of generator polynomials of $\lam$-constacyclic   Hermitain complementary dual codes of length $n$ over $\Fqq$ depend only on 
	$u_{ij}\in\left\{0, p^\nu\right\}$ and $(v_{ij}, w_{ij})\in\left\{(0,0), (p^\nu,p^\nu) \right\}$.
	Then the number of   $\lam$-constacyclic Hermitain complementary dual codes of length $n$ over $\Fqq$ is $2^{\mathtt{s}+\mathtt{t}}$.
\end{proof}

\section{MDS  Constacyclic Hermitian Self-dual Codes over $\Fqq$}
In this section, we construct a class of MDS $\lambda$-constacyclic  Hermitian self-dual codes over $\Fqq$. Throughout this section, let $n$ be an even positive integer relatively prime to $q$ such that $(nr)| (q^2-1)$ and $r|(q+1)$, where $r=o_{q^2}(\lambda)$. Equivalently,  $n={n^\prime}$ in the previous section.

In \cite{Yang13}, a family of MDS constacyclic Hermitian self-dual code over $\Fqq$ whose length is a divisor of $q-1$ is introduced. We now introduce a new family of MDS constacyclic Hermitian self-dual code over $\Fqq$ whose length is a divisor of $q+1$. Therefore, our family is different to a family in \cite{Yang13} if $n\neq 2$.

Let $\xi$ be a primitive $nr$th root of unity in an extension field $\F_{q^{2k}}$ of $\Fqq$ such that $\xi^n=\lambda$. Then the set of all roots of $x^n-\lam$ is $\left\{\xi, \xi^{r+1}, \xi^{2r+1},\ldots, \xi^{(n-1)r+1} \right\}$. Define 
$$O_{r,n}=\left\{1, r+1, 2r+1,\ldots, (n-1)r+1 \right\}=\left\{ir+1\mid 0\leq i\leq n-1\right\}\subseteq \Z_{nr}.$$
Let $C$ be a $\lam$-constacyclic code. The \textit{roots} of the code $C$ is defined to be the roots of its generator polynomial.
The {\em defining set of $\lam$-constacyclic code} $C$ is defined as $$T:=\left\{ ir+1\in O_{r,n}\mid \xi^{ir+1} \text{~is a root of~} C \right\}.$$ It is not difficult to see that $T\subseteq O_{r,n}$ and $\dim C=n-|T|$.
The following theorem can be obtained by slightly modified \cite[Corollary 3.3]{Yang13}.
\begin{thm}\label{hsdtest}
	Let $C_T$ be a $\lam$-constacyclic code with the defining set $T$. Then
	\begin{enumerate}[(i)]
		\item $C_T$ is a Hermitian self-orthogonal constacyclic code if and only if $O_{r,n}\smallsetminus T\subseteq -qT$.
		\item $C_T$ is a  constacyclic Hermitian self-dual code if and only if $-qT=O_{r,n}\smallsetminus T$, or equivalently, $T\cap -qT=\emptyset$.
	\end{enumerate}
\end{thm}
\begin{proof}
	Note that $q^2\equiv 1\,{\rm mod}\, nr$. Then, by \cite[Corollary 3.3]{Yang13}, $C_T$ is  a Hermitian self-orthogonal constacyclic code if and only if $-q\left(O_{r,n}\smallsetminus T\right)\subseteq T$. Hence, \[O_{r,n}\smallsetminus T=-q\left(-q\left(O_{r,n}\smallsetminus T\right)\right)\subseteq -qT.\] The proof of (ii) can be obtained similarly.
\end{proof}

The BCH bound for constacyclic codes is as follows.
\begin{thm}[{\cite[Theorem 2.2]{Aydin2001}}]\label{bchbdd}
	Let $C$ be a $\lambda$-constacyclic code of length $n$ over $\Fqq$. Let $r=o_{q^2}(\lambda)$. Let $\xi$ be a primitive $nr${\rm th} root of unity in an extension field of $\Fqq$ such that $\xi^{n}=\lambda$. Assume the generator polynomial of $C$ has roots that include the set $\left\{\xi^{ri+1}\mid i_1\leq i\leq i_1+d-1 \right\}$. Then the minimum distance of $C$ is at least $d+1$. 
\end{thm}
\begin{ex}\label{ex11}
	Let $q=3,~n=4$ and let $\lam=-1$ in $\F_9$.  Then $o_{9}(\lam)=2$ and $O_{2,4}=\{1, 3, 5, 7\}$. Let $T=\{1, 3\}$. Then $-qT=5T=\{5, 7\}$. 
	By Theorems \ref{hsdtest}-\ref{bchbdd} and the Singleton bound, $C_T$ is an MDS $\lambda$-constacyclic Hermitian self-dual code with parameter $[4, 2, 3]$ over $\F_9$.
\end{ex}
\begin{ex}\label{ex12}
	Let $q=11,~n=6$ and let $\lam=\al^{30}$ in $\F_{121}$ where $\al$ is a primitive element of $\F_{121}$. Then $o_{121}(\al^{30})=4$ and $O_{4,6}=\{1, 5, 9, 13, 17, 21\}$. Let $T=\{1, 5, 9\}$. Then $-qT=13T=\{13, 17, 21\}$. 
	By Theorems \ref{hsdtest}-\ref{bchbdd} and the Singleton bound, $C_T$ is an MDS  $\lam$-constacyclic Hermitian self-dual code with parameter $[6, 3, 4]$ over $\F_{121}$.
\end{ex}

Examples \ref{ex11} and \ref{ex12} show that there exist MDS  constacyclic Hermitian self-dual codes. The following theorem is a generalization of Examples \ref{ex11} and \ref{ex12}.

\begin{thm}\label{tmds}
	Let $\lam\in\Fqq$ be such that $r=o_{q^2}(\lambda)$ is even. Let $n$   be an  even integer such that $nr|(q^2-1)$ both  $n$ and $r$ divide $q+1$. Let $$T=\left\{ 1, r+1,\ldots,\left(\frac{n}{2}-1\right)r+1\right\}=\left\{ir+1\mid 0\leq i\leq \frac{n}{2}-1 \right\}.$$ If $\frac{2(q+1)}{nr}$ is odd, then $C_T$ is an MDS $\lambda$-constacyclic  Hermitian self-dual  code with parameters $\left[n,\frac{n}{2},\frac{n}{2}+1\right]$.
\end{thm}
\begin{proof}
	Note that $$O_{r,n}=\left\{1, r+1, 2r+1,\ldots, (n-1)r+1 \right\}=\left\{ir+1\mid 0\leq i\leq n-1\right\}\subseteq \Z_{nr}$$ and $$O_{r,n}\smallsetminus T=\left\{\left(\frac{n}{2}\right)r+1, \left(\frac{n}{2}+1\right)r+1,\ldots, (n-1)r+1 \right\}=\left\{\left(\frac{n}{2}+i\right)r+1\mid 0\leq i\leq\frac{n}{2} \right\}.$$ Claim that $-qT=O_{r,n}\smallsetminus T$ such that $-q(ir+1)=1+(\frac{n}{2}+i)r$ for all $0\leq i\leq \frac{n}{2}$. Since $\frac{2(q+1)}{nr}$ is odd, $\frac{q+1}{r}+\frac{n}{2}\equiv 0\,{\rm mod}\, n$. Then $\frac{q+1}{r}+(\frac{n}{2}+i+iq)\equiv \frac{q+1}{r}+\frac{n}{2}+i(q+1)\equiv 0\,{\rm mod}\, n$. We obatin $q+1+\left(\frac{n}{2}+i+iq\right)r\equiv \left(\frac{q+1}{r}\right)r+\left(\frac{n}{2}+i(q+1)\right)r\equiv 0\,{\rm mod}\, nr$, or equivalently, $-q(ir+1)\equiv 1+\left(\frac{n}{2}+i\right)r\,{\rm mod}\, nr$. Therefore, the code $C_T$ is a $\lambda$-constacyclic Hermitian self-dual  code. Clearly,  $C$ is an MDS $\lambda$-constacyclic  code with parameters $\left[n,\frac{n}{2},\frac{n}{2}+1\right]$.
	
\end{proof}

Since the length of the MDS codes given in  \cite{Yang13}  is a divisor of $q-1$ and the MDS codes condtructed in Theorem \ref{tmds}  is a divisor of $q+1$, the later is different from the former whenever $n\ne 2$.  Some families of codes derived from Theorem \ref{tmds} are given in the following example.

\begin{ex} Let $q$  be an odd prime and $m$ be the largest positive integer such that $q\equiv -1 \,{\rm mod}\, 2^m$.  For each $1\leq i\leq m-1$,  let $r=2^i$ and $n= \frac{q+1}{2^{m-i}}$. Then $n|(q+1)$ and $\frac{2(q+1)}{nr}=\frac{q+1}{r^{m}}$ is odd.   Therefore, by Theorem \ref{tmds},  there exists an MDS $[n= \frac{q+1}{2^{m-i}},  \frac{q+1}{2^{m-i+1}},\frac{q+1}{2^{m-i+1}}+1 ]$ code over $\mathbb{F}_{q^2}$ for all  $1\leq i\leq m-1$.
\end{ex}
Conditions for nonexistence MDS $\lambda$-constacyclic  Hermitian self-dual  codes of length $n$ over $\mathbb{F}_{q^2}$ are given as follows.
\begin{thm}\label{nonhsd}
	Let $\lam\in\Fqq$ be such that $r=o_{q^2}(\lambda)$ is even. Let $n$   be an even integer such that $nr|(q^2-1)$. If $a\left(\frac{q+1}{r}\right)\equiv 0\,{\rm mod}\, n$ for some $a$ in $O_{r,n}$ then there are no MDS $\lambda$-constacyclic Hermitian self-dual   codes of length $n$ over $\Fqq$.
\end{thm}
\begin{proof}
	Let $a$ in $O_{r,n}$ be such that $a\left(\frac{q+1}{r}\right)\equiv 0\,{\rm mod}\, n$. Then, $a(q+1)\equiv 0\,{\rm mod}\, nr$, which implies $-qa=a$.
	Let $C_T$ be an MDS $\lambda$-constacyclic  Hermitian self-dual code and let $T\subseteq O_{r,n}$ be the defining set of a code $C_T$. By Theorem \ref{hsdtest}, $T\cap -qT=\emptyset$ and $T\cup -qT=O_{r,n}$.
	
	If $a\in T$, then $a\in T\cap -qT$, a contradiction. If $a\not\in T$, then $a\in -qT$. So $a=-qa\in q^2T=T$, a contradiction. 
\end{proof}

The following example shows that there are no ($-1$)-constacyclic Hermitian self-dual code of length $6$ over $\F_{49}$.
\begin{ex}
	Let $q=7,~n=6$ and let $\lam=-1$ in $\F_{49}$. Thus, $o_{49}(-1)=2$ and $O_{2,6}=\left\{1, 3, 5, 7, 9, 11\right\}$. Since
	$3\cdot \frac{8}{2}=12\equiv 0 \,{\rm mod}\,6,$
	by Theorem \ref{nonhsd}, there are no MDS ($-1$)-constacyclic Hermitian self-dual code of length $6$ over $\F_{49}$.
\end{ex}
\section{Quasi-Twisted Hermitian Self-Dual Codes over $\Fqq$}
In this section, we focus on simple root $(\lam, \ell)$-QT Hermitian self-dual codes of length $n\ell$ over $\Fqq$, or equivalently,  $\gcd(n,q)=1$. The decomposition of $(\lam, \ell)$-QT  codes  is  given. The characterization and enumeration of  $(\lam, \ell)$-QT Hermitian self-dual codes of length $n\ell$ over $\Fqq$ can be obtained via this decomposition.

In \cite{Jia}, QT codes over finite fields with respect to the Euclidean inner product were studied. QT codes were decomposed and the Euclidean duals of such codes are determined. In particular, the characterization of Euclidean self-dual QT codes were given. As a generalization of \cite{Jia} and  \cite{LingI}, we study QT codes over finite fields with respect to the Hermitian inner product.. 

From Lemma \ref{eqipp},  every $(\lam, \ell)$-QT code of length $n\ell$ over $\Fqq$ can be viewed as  an $R$ submodule of $R^\ell$, where   $R:={\Fqq[x]}/{\left\langle x^n-\lam\right\rangle}$.

Define an \textit{involution} $\sim$ on $R$ to be the $\Fqq$-linear map that sends $\al$ to $\al^q$ for all $\al\in\Fqq$ and sends $x$ to $x^{-1}=x^{n-1}$. Let $\left\langle \cdot,\cdot\right\rangle_\sim:R^\ell\times R^\ell\rightarrow R$ be defined by 
$$\left\langle \textbf{v},\textbf{s}\right\rangle_\sim:=\sum_{j=0}^{\ell-1}v_j(x)\widetilde{s_j(x)},$$ where $\textbf{v}=(v_0(x),\ldots,v_{\ell-1}(x))$ and $\textbf{s}=(s_0(x),\ldots,s_{\ell-1}(x))$ in $R^\ell$. The \textit{$\sim$-dual} of $D\subseteq R^\ell$ is defined to be the set
\[D^{\perp_\sim}:=\left\{\textbf{v}\in R^\ell \,\left\vert\, \left\langle \textbf{v},\textbf{s}\right\rangle_\sim=0 ~\text{ for all}~ \textbf{s}\in D \right.\right\}.\]
We say that $D\subseteq R^\ell$ is \textit{$\sim$-self-dual} if $D=D^{\perp_\sim}$.

\begin{prop}
	Let $\textbf{a}, \textbf{b}\in\Fqq^{n\ell }$. Then $\left\langle T_{\lam,\ell}^k(\textbf{a}),\textbf{b}\right\rangle_H=0$ for all $0\leq k\leq n-1$ if and only if $\left\langle \psi(\textbf{a}),\psi(\textbf{b})\right\rangle_\sim=0$.
\end{prop}
\begin{proof}
	Let $\psi(\textbf{a})=\left(a_0(x),\ldots,a_{\ell-1}(x)\right)=\left(\sum_{i=0}^{n-1}a_{i0}x^i,\ldots,\sum_{i=0}^{n-1}a_{i, \ell-1}x^i\right)$ and\\ $\psi(\textbf{b})=\left(b_0(x),\ldots,b_{\ell-1}(x)\right)=\left(\sum_{i=0}^{n-1}b_{i0}x^i,\ldots,\sum_{i=0}^{n-1}b_{i, \ell-1}x^i\right)$. By comparing the coefficients,
	\begin{align}\label{cc}
	0&=\sum_{j=0}^{\ell-1}a_j(x)\widetilde{b_j(x)} =\sum_{j=0}^{\ell-1}\left(\sum_{i=0}^{n-1}a_{ij}x^i\right)\left(\sum_{k=0}^{n-1}b_{kj}^qx^{-k}\right)
	\end{align}
	is equivalent to
	\begin{align}\label{dd}
	\sum_{j=0}^{\ell-1}\sum_{i=0}^{n-1}a_{i+h,j}b_{ij}^qx^h=0
	\end{align}
	for all $0\leq h\leq n-1$, where the subscripts $i+h$ are computed modulo $n$. The expression in \eqref{dd} is equivalent to $\left\langle T_{\lam,\ell}^{-h}(\textbf{a}), \textbf{b}\right\rangle_H=0$ for all $0\leq h\leq n-1$. Since $T_{\lam,\ell}^{-h}=T_{\lam,\ell}^{(n-h)}$, \eqref{cc} is equivalent to $\left\langle T_{\lam,\ell}^k(\textbf{a}),\textbf{b}\right\rangle_H=0$ for all $0\leq k\leq n-1$.
\end{proof}

Next proposition follows  from the definition of QT codes and \cite[Proposition 2.3]{Yang13}.
\begin{prop} \label{prop20}
	Let $C$ be a $(\lam,\ell)$-QT code of length $n\ell $ over $\Fqq$ and let $C^{\perp_H}$ be the Hermitian dual of $C$. Then $C^{\perp_H}$ is a $(\lam^{-q},\ell)$-QT code of length $n\ell $ over $\Fqq$.
\end{prop}
\begin{proof}
	Let $\textbf{d}\in C^{\perp_H}$ and let $\textbf{c}\in C$. Then $\left\langle T_{\lam,\ell}^i(\textbf{c}),\textbf{d}\right\rangle_H=0$ for all $0\leq i\leq n-1$ .
	Since 
	\begin{align*}
	\left\langle\textbf{c} ,T_{\lam^{-q},\ell}(\textbf{d})\right\rangle_H=~&\left\langle\textbf{c} ,T_{\lam^{-1},\ell}(\textbf{d}^q)\right\rangle_E \text{~~where~~} \textbf{d}^q \text{~~denote~~} \left(d_{00}^q,\ldots,d_{n-1,\ell-1}^q\right)\\
	=~&\sum_{j=0}^{\ell-1}c_{0j}d_{n-1,j}^q\lam^{-1}+\sum_{i=1}^{n-1}\sum_{j=0}^{\ell-1}c_{ij}d_{i-1,j}^q\\
	=~&\lam^{-1}\left(\sum_{j=0}^{\ell-1}c_{0j}d_{n-1,j}^q+\sum_{i=1}^{n-1}\sum_{j=0}^{\ell-1}\lam c_{ij}d_{i-1,j}^q \right)\\
	=~&\lam^{-1}\left\langle T_{\lam,\ell}^{n-1}(\textbf{c}), \textbf{d}\right\rangle_H\\
	=~&0,
	\end{align*}
	it follows that 
	$T_{\lam^{-q},\ell}(\textbf{d})\in C^{\perp_H}$. Therefore, $C^{\perp_H}$ is a $(\lam^{-q},\ell)$-QT code.
\end{proof}
By Proposition \ref{prop20},  both $C$ and $C^{\perp_H}$ are $(\lam,\ell)$-QT codes if and only if $o_{q^2}(\lam)|(q+1)$. Therefore, it makes sense to focus on only the case where $o_{q^2}(\lam)|(q+1)$.
\begin{cor}
	Let $\lam\in\Fqq\smallsetminus\{0\}$ be such that $o_{q^2}(\lam)|(q+1)$. Let $C$ be a $(\lam, \ell)$-QT code of length $n\ell$ over $\Fqq$ and let $\psi(C)$ be its image in $R^\ell$ under $\psi$ defined in (\ref{psi}). Then $\psi(C)^{\perp_\sim}=\psi(C^{\perp_H})$. In particular, $C$ is Hermitian self-dual if and only if $\psi(C)$ is $\sim$-self-dual.
\end{cor}

\subsection{Decomposition}
Since $\gcd(q,n)=1$,  by \eqref{xn-1},  $x^n-\lam$ can be factored   as follows
$$x^n-\lam=g_1(x)\ldots g_s(x)h_1(x)h_1^\dagger(x)\ldots h_t(x)h_t^\dagger(x),$$
where $h_j(x)$ and $h_j^\dagger(x)$ are a monic irreducible conjugate-reciprocal polynomial pair for all  $1\leq j\leq t$ and $g_i(x)$ is a monic irreducible self-conjugate-reciprocal polynomial for all  $1\leq i\leq s$.

By the Chinese Remainder Theorem (c.f. \cite{Jia} and \cite{LingI} ), we write
$$R={\Fqq[x]}/{\left\langle x^n-\lam\right\rangle}\cong\left(\prod_{i=1}^s{\Fqq[x]}/{\left\langle g_i(x)\right\rangle} \right)\times\left(\prod_{j=1}^t\left({\Fqq[x]}/{\left\langle h_j(x)\right\rangle}\times {\Fqq[x]}/{\left\langle h_j^\dagger(x)\right\rangle} \right)\right).$$

For simplicity,  let  $G:=\Fqq[x]/\left\langle g_i(x)\right\rangle_i,$ $H'_j:=\Fqq[x]/\left\langle h_j(x)\right\rangle$ \ and   $H''_j:=\Fqq[x]/\left\langle h_j^\dagger(x)\right\rangle$.   Then we have
\begin{equation}\label{h'h''}
R={\Fqq[x]}/{\left\langle x^n-\lam\right\rangle}\cong\left(\prod_{i=1}^sG_i\right)\times\left(\prod_{j=1}^t\left(H'_j\times H''_j\right)\right).
\end{equation}

For an irreducible self-conjugate-reciprocal factor $f(x)$  of $x^n-\lam$ in $\Fqq[x]$ of degree $k$, the map ${\Fqq[x]}/{\left\langle f(x)\right\rangle}$ 
$$\bar{ }:{\Fqq[x]}/{\left\langle f(x)\right\rangle}\longrightarrow {\Fqq[x]}/{\left\langle f(x)\right\rangle}$$
defined by ${c}(x)=\sum_{i=0}^{k-1}c_ix^i+\left\langle f\right\rangle\mapsto \overline{c(x)}=\sum_{i=0}^{k-1}c_i^qx^{-i}+\left\langle f\right\rangle$ is an automorphim.

For irreducible conjugate-reciprocal factors pair $f(x)$ and $f^\dagger(x)$ of $x^n-\lam$
in $\Fqq[x]$ of degree $k$, the extension fields ${\Fqq[x]}/{\left\langle f(x)\right\rangle}$ and ${\Fqq[x]}/{\left\langle f^\dagger(x)\right\rangle}$
are isomorphic. The map
$$\hat{ }:{\Fqq[x]}/{\left\langle f(x)\right\rangle}\longrightarrow {\Fqq[x]}/{\left\langle f^\dagger(x)\right\rangle}$$
defined by ${c}(x)=\sum_{i=0}^{k-1}c_ix^i+\left\langle f(x) \right\rangle\mapsto \widehat{c(x)}=\sum_{i=0}^{k-1}c_i^qx^{-i}+\left\langle f^\dagger(x)\right\rangle$ is an isomorphism.

Using the above isomorphisms, we have 
\begin{equation}\label{h'h'}
R={\Fqq[x]}/{\left\langle x^n-\lam\right\rangle}\cong\left(\prod_{i=1}^sG_i\right)\times\left(\prod_{j=1}^t\left(H'_j\times H'_j\right)\right).
\end{equation}
Let $\sigma_1, \sigma_2$ denote the isomorphisms in \eqref{h'h''} and \eqref{h'h'}, respectively. Then an element $\textbf{r}\in R$ can be written as 
$$\sigma_1(\textbf{r})=\left(r_1,\ldots,r_s,r_1',r_1'',\ldots,r_t',r_t''\right) \text{~~in~~} \left(\prod_{i=1}^sG_i\right)\times\left(\prod_{j=1}^t\left(H'_j\times H''_j\right)\right),$$
where $r_i\in G_i, r_j'\in H_j'$ and $r_j''\in H_j''$, and
\begin{equation}\label{aa}
\sigma_2(\textbf{r})=\left(r_1,\ldots,r_s,r_1',\widehat{r_1''},\ldots,r_t',\widehat{r_t''}\right) \text{~~in~~} \left(\prod_{i=1}^sG_i\right)\times\left(\prod_{j=1}^t\left(H'_j\times H'_j\right)\right),
\end{equation}
where $r_i\in G_i$ and $r_j', \widehat{r_j''}\in H_j'$.
Therefore, an element $\tilde{\textbf{r}}\in R$ can be expressed as
$$\sigma_1(\tilde{\textbf{r}})=\left(\overline{r_1},\ldots,\overline{r_s},\widehat{r_1''},\widehat{r_1'},\ldots,\widehat{r_t''},\widehat{r_t'}\right) \text{~~in~~} \left(\prod_{i=1}^sG_i\right)\times\left(\prod_{j=1}^t\left(H'_j\times H''_j\right)\right),$$
where $\overline{r_i}\in G_i, \widehat{r_j''}\in H_j'$ and $\widehat{r_j'}\in H_j''$,  and
\begin{align}\label{bb}
\sigma_2(\tilde{\textbf{r}})=\left(\overline{r_1},\ldots,\overline{r_s},\widehat{r_1''},{r_1'},\ldots,\widehat{r_t''},{r_t'}\right) \text{~~in~~} \left(\prod_{i=1}^sG_i\right)\times\left(\prod_{j=1}^t\left(H'_j\times H'_j\right)\right),
\end{align}
where $\overline{r_i}\in G_i$ and $ \widehat{r_j''}, {r_j'}\in H_j'$.
\begin{rem}
	If  $f(x)$ is self-conjugate-reciprocal, then $\sim$ induces the field automorphism\, $\bar{ }$ on $\Fqq[x]/\left\langle f(x)\right\rangle\cong \F_{q^{2k}}$.  The map $r\mapsto \overline{r}$ on $\Fqq[x]/\left\langle f(x)\right\rangle$ is actually the map $r\mapsto r^{q^k}$ on $\F_{q^{2k}}$.
\end{rem}

Using statements similar to those in the proof of \cite[Proposition 4.1]{LingI}, we conclude the next proposition.
\begin{prop}\label{ortho}
	Let $\textbf{a}, \textbf{b}\in R^\ell$ and write
	$\textbf{a}=(\textbf{a}_0, \textbf{a}_1,\ldots,\textbf{a}_{\ell-1})$ and $\textbf{b}=(\textbf{b}_0, \textbf{b}_1,\ldots,\textbf{b}_{\ell-1})$. 
	Decomposing $\sigma_2(\textbf{a}_i)$ and $ \sigma_2(\textbf{b}_i)$ using \eqref{h'h'}, we have 
	$$\sigma_2(\textbf{a}_i)=(a_{i1},\ldots,a_{is},a_{i1}',a_{i1}'',\ldots,a_{it}',a_{it}'') 
	\text{ ~~ and ~~ }
	\sigma_2(\textbf{b}_i)=(b_{i1},\ldots,b_{is},b_{i1}',b_{i1}'',\ldots,b_{it}',b_{it}''),$$
	where $a_{ij}, b_{ij}\in G_i$ and $a_{ij}', a_{ij}'', b_{ij}', b_{ij}''\in H_j'$. Then
	\begin{align*}
	\left\langle \sigma_2(\textbf{a}), \sigma_2(\textbf{b})\right\rangle_\sim&=\sum_{i=1}^{\ell-1}\sigma_2(\textbf{a}_i)\widetilde{\sigma_2(\textbf{b}_i)}\\
	&=\left(\sum_i a_{i1}\overline{b_{i1}},\ldots,\sum_ia_{is}\overline{b_{is}},\sum_i a_{i1}'b_{i1}'', \sum_ia_{i1}''b_{i1}',\ldots,\sum_ia_{it}'b_{it}'', \sum_ia_{it}''b_{it}'\right).
	\end{align*}
	In particular, $\left\langle \sigma_2(\textbf{a}), \sigma_2(\textbf{b})\right\rangle_\sim=0$ if and only if $\sum_ia_{ij}\overline{b_{ij}}=0$ for all $1\leq j\leq s$ and 
	$\sum_i a_{ik}'b_{ik}''=0=\sum_ia_{ik}''b_{ik}'$ for all $1\leq k\leq t$.
\end{prop}
By \eqref{h'h'}, we have
\begin{equation*}
R^\ell\cong\left(\prod_{i=1}^sG_i^\ell\right)\times\left(\prod_{j=1}^t\left({H'_j}^\ell\times {H'_j}^\ell\right)\right).
\end{equation*}
In particular, $R$ submodule $C$ of $R^\ell$ can be decomposed as
$$C\cong \left(\prod_{i=1}^sC_i\right)\times\left(\prod_{j=1}^t\left({C'_j}\times {C''_j}\right)\right),$$
where  $C_j'$ and $C_j''$ are linear codes of length $\ell$ over $H'_j$ and $C_i$ is a linear code of length $\ell$ over $G_i$.

By Proposition \ref{ortho}, we have 
$$C^{\perp_{H}}\cong \left(\prod_{i=1}^sC_i^{\perp_{H}}\right)\times\left(\prod_{j=1}^t\left({{C''_j}^{\perp_{E}}}\times {{C'_j}^{\perp_{E}}}\right)\right),$$
and hence, the next corollary follows.
\begin{cor}
	An $R$ submodule $C$ of $R^\ell$ is $\sim$-self-dual, or equivalently, a $(\lam,\ell)$-QT code $\psi^{-1}(C)$ of length $n\ell$ over $\Fqq$ is Hermitian self-dual if and only if
	$$C\cong \left(\prod_{i=1}^sC_i\right)\times\left(\prod_{j=1}^t\left({C'_j}\times {C'_j}^{\perp_E}\right)\right),$$
	where $C_i$ is a Hermitian self-dual code of length $\ell$ over $G_i$ for $1\leq i\leq s$,  $C_j'$ is a linear code of length $\ell$ over $H_j$, and $C_j'^{\perp_E}$ is  Euclidean dual of   $C_j'$ for $1\leq j\leq t$.
\end{cor}
Let $N(q,\ell)$ (resp., $N_H(q,\ell)$) denote the number of linear codes (resp., Hermitian self-dual codes) of length $\ell$ over $\Fq$.  It is well known \cite{Sloane} that
\begin{align}
N(q,\ell)&=\sum_{i=0}^\ell\prod_{j=0}^{i-1}\frac{q^{\ell}-q^{i}}{q^{i}-q^{j}}
\end{align}
and 
\begin{align}
N_H(q, \ell)&=\begin{cases}\displaystyle\prod_{i=0}^{\frac{\ell}{2}-1}(q^{i+\frac{1}{2}}+1) &\text{ if }  \ell \text{ is even} ,\\
0, &\text{  otherwise},\end{cases}
\end{align}
where the empty product is regarded as $1$.
\begin{prop}
	Let $\Fqq$ be a finite field and let $n, \ell$ be positive integers such that $\ell$ is even and $\gcd(n,q)=1$. Let $\lam$ be a nonzero element in $\Fqq$ such that $o_{q^2}(\lam)|(q+1)$. Suppose that $x^n-\lam=g_1(x)\ldots g_s(x) h_1(x)h_1^\dagger(x)\ldots h_t(x)h_t^\dagger(x)$. Let $d_i\deg g_i(x)$ and $e_j=\deg h_j(x)$. The number of $(\lam,\ell)$-QT Hermitian self-dual  codes of length $n\ell $ over $\Fqq$ is $$\prod_{i=1}^sN_H(q^{2d_i},\ell) \prod_{j=1}^t N(q^{2e_j},\ell).$$  
\end{prop}
In the case where $\Omega =\emptyset$ or $\pi({n^\prime}r, q^2)=0$, the formula for the  number of Hermitian self-dual $(\lam,\ell)$-QT codes of length $n\ell $ over $\Fqq$ can be simplified in the following corollaries.
\begin{cor}Let $x^n-\lam=g_1(x)\ldots g_s(x) h_1(x)h_1^\dagger(x)\ldots h_t(x)h_t^\dagger(x)$. and let $e_j=\deg h_j(x)$.
	If $\Omega=\emptyset$, then  the number of  $(\lam,\ell)$-QT Hermitian self-dual codes of length $n\ell$ over $\Fqq$ is
	$$\prod_{j=1}^t N(q^{2e_j},\ell).$$
\end{cor}

\begin{cor}Let $x^n-\lam=g_1(x)\ldots g_s(x) h_1(x)h_1^\dagger(x)\ldots h_t(x)h_t^\dagger(x)$ and let $d_i=\deg g_i(x)$.
	If  $\pi({n^\prime}r, q^2)=0$, then  the number of $(\lam,\ell)$-QT  Hermitian self-dual codes of length $n\ell $ over $\Fqq$ is
	$$\prod_{i=1}^sN_H(q^{2d_i},\ell).$$
\end{cor}
%
%

\section*{Acknowledgements}
	The authors thank   Borvorn Suntornpoch for useful discussions.



%

\end{document}